%% file: ex_article.tex
\documentclass[review,hidelinks,onefignum,onetabnum]{siamart250211}
\usepackage{tikz}
\usepackage[table,xcdraw]{xcolor}
\usepackage{multirow}
\usepackage{multicol}
\usepackage{enumitem}
\usepackage{subcaption}

\input{ex_shared}

\ifpdf
\hypersetup{
  pdftitle={Third-Order Iterative Algorithms},
  pdfauthor={K. D. Prabhu, S. Singh and V. A. Vijesh}
}
\fi

\begin{document}
\maketitle
\begin{abstract}
This paper develops an efficient iterative method for computing all zeros of solutions to second-order ordinary differential equations (ODEs). Approximating the solution of an associated Riccati differential equation using the trapezoidal rule leads to the well-known Halley's method. To improve computational efficiency, a modification to Halley's method is proposed. This modification ensures that, beyond the first iteration, the proposed method requires only an evaluation of the function and its first derivative, like Newton's method. The modified Halley's method also retains third-order convergence for solutions of second-order ODEs. Based on the behaviour of the coefficients of the second-order ODE, nonlocal convergence results are established for both Halley's and modified Halley's methods. This approach also provides suitable initial guesses for both methods to compute all zeros of solutions to second-order ODEs in a given interval under suitable assumptions. As applications, this manuscript presents intriguing algorithms based on the modified Halley's method to compute all nodes and weights for Gauss–Legendre and Gauss–Hermite quadratures. A comparative numerical study with recent methods demonstrates the efficiency of the proposed algorithms.
\end{abstract}

\section{\bf Introduction}
Developing computationally efficient algorithms for finding zeros of a function, which are solutions to a second-order differential equation, is an important problem \cite{Br17, GLR07, Se10, SV90}. Many well-known special functions and various classical orthogonal polynomials are solutions of the second-order linear differential equations. Thus, studying the zeros of solutions to second-order differential equations provides a unified approach to a class of special functions and orthogonal polynomials. This manuscript 
presents computationally efficient iterative algorithms for finding the zeros of a nontrivial solution to the following second-order differential equation:
\begin{equation}\label{ch5_f_ode}
		f''(x)+r(x)f(x)=0,
\end{equation}
where $r(x)$ is a continuous function.  It is worth mentioning that as $f(x)$ is a nontrivial solution of \eqref{ch5_f_ode}, $f(x)$ and $f'(x)$ do not have common zeros. Let $\phi(x)=x-\frac{f(x)}{f'(x)}$ be the Newton iteration function corresponding to the function $f(x)$. It is interesting to note that if $x^*$ is a zero of $f(x)$, then $\phi(x^*)=x^*$, $\phi'(x^*)=0$, and $\phi''(x^*)=0$. Consequently, Newton's method has third-order convergence for functions that are solutions of \eqref{ch5_f_ode}.

Though Newton's method has higher-order convergence for $f(x)$, it has two major problems when $r(x)$ is positive. Let $x^*$ be the unique zero of $f(x)$ in an interval $\mathcal{I}$ such that $\mathcal{I}$ does not have any zero of $f'(x)$. Hence, if $r(x)>0$ in $\mathcal{I}$, then $f(x)$ and $f''(x)$ always have opposite signs in $\mathcal{I}$. Consequently, the well-known Fourier condition\cite[Theorem 1.1]{Me97} for nonlocal convergence for Newton's method fails for all the points in the interval $\mathcal{I}$. Consequently, Newton's method may not have nonlocal convergence in $\mathcal{I}$. Hence, one has to find an initial guess sufficiently close to $x^*$ to ensure the convergence of Newton's method. Thus, finding a suitable initial guess for Newton's method is difficult in this scenario. If $r(x)>0$, in the real line under suitable additional conditions on $r(x)$, one can see that $f(x)$ has an infinite number of zeros on the real line. Hence, if an interval has more than one zero of $f(x)$, then $f'(x)$ has at least one zero in that interval. Consequently, to apply Newton's method, one has to find a suitable interval for each zero such that it does not have any zero of $f'(x)$. This step is also crucial for any theoretical convergence analysis of Newton's method. 

Literature typically tackles this issue by utilising other mathematical tools to obtain an initial guess that is sufficiently close to a specific zero. In this direction, the asymptotic method is one of the tools frequently used in the literature (for example \cite{BMF12, HT13, TTO16}). In this approach, using the asymptotic behaviour of the zeros of a specific function $f(x)$, an approximate zero is obtained. Later, the approximate zero is refined using Newton's method. Algorithms in \cite{BMF12, HT13, TTO16} computed all the zeros of Legendre \cite{BMF12}, Jacobi \cite{HT13}, and Hermite \cite{TTO16} polynomials by improving the approximate zeros obtained from the asymptotic method as an initial guess for Newton's method. It is worth mentioning that the initial guess in \cite{HT13} for the zeros of Jacobi polynomials requires the zeros of the Bessel function. Similarly, the initial guess in \cite{TTO16} for the zeros of Hermite polynomials requires the zeros of Airy's function and the zeros of a nonlinear function involving the trigonometric function \textit{sine}. An interesting algorithm developed in \cite{GLR07} for finding all the zeros of solutions to linear homogeneous second-order ODEs. In this approach, the initial guesses for Newton's method are obtained by numerically solving a Cauchy problem associated with the second-order differential equation. In particular, using the Pr\"ufer transformation, an associated Cauchy problem is obtained first, and then the Runge–Kutta method is used to solve the Cauchy problem numerically. The initial guess for a specific zero depends on its predecessor; special attention is needed to find the first zero. The usage of Newton's method is also found in computing zeros of solutions of second-order ODEs via the phase function technique \cite{Br17, Ol50, SV90}. This procedure employs Newton's method as a tool to invert the phase function at selected points. 

Instead of finding a suitable initial guess for Newton's method, an alternate concept is used in \cite{GS03,  PSV25, Se02, Se10} to determine all the zeros of solutions to a second-order differential equation. In this concept, Newton's method is replaced by an iterative method with better non-local convergence properties. The other advantage is that this concept produces simple initial guesses for the iterative method that has theoretical convergence. In \cite{GS03, Se02}, an intriguing second-order iterative method was proposed for finding zeros of a function $y(x)$, that is, a solution to a second-order linear ODE. In the first step, a nice function $w(x)$ is identified with the property that the zeros of $y(x)$ and $w(x)$ are interlaced. As the functions $y(x)$ and $h(x)=\frac{y(x)}{w(x)}$ have the same zeros, an iterative method was developed in \cite{GS03, Se02} to find the zeros of $h(x)$. It was shown that $h(x)$ satisfies a Riccati differential equation. The second-order iterative method is then obtained by solving the associated Riccati differential equation semi-analytically. Sufficient conditions were derived to ensure the non-local convergence of the iterative method between the successive singular points of $h(x)$. Using bounds for the distance between the zeros of $y(x)$ and $w(x)$, a specific simple initial guess is provided to ensure convergence to a specific zero. Recently, by employing the trapezoidal rule to solve the associated Riccati differential equation, a third-order iterative method was developed in \cite{PSV25}. Similar to \cite{GS03, Se02}, sufficient conditions were derived in \cite{PSV25} to ensure the non-local convergence of the third-order iterative method between the successive singular points of $h(x)$.

In \cite{Se10}, an interesting fourth-order iterative scheme is derived to compute all the zeros of $f(x)$. As the functions $f(x)$ and $h(x)=\frac{f(x)}{f'(x)}$ have the same zeros, an iterative method was developed in \cite{Se10} to find the zeros of $h(x)$. Using \eqref{ch5_f_ode}, an associated Riccati differential equation for $h(x)$ is obtained first. The fourth-order iterative method is then obtained by solving the associated Riccati differential equation semi-analytically. If $r(x)>0$, then the zeros and singular points of $h(x)$ are interlaced. Sufficient conditions were derived to ensure the non-local convergence of the iterative method between the successive singular points. Using bounds for the distance between the zeros of $f(x)$ and $f'(x)$, a specific simple initial guess is provided to ensure convergence to a specific zero. Recently, this fourth-order method has been employed for finding nodes and weights for the Gauss-Hermite \cite{GST19}, Gauss-Laguerre \cite{GST19}, and Gauss-Jacobi quadratures \cite{GST21a}.

In this manuscript, an iterative method is derived by solving the Riccati differential equation in \cite{Se10} using the Trapezoidal Rule, similar to the approach in \cite{PSV25}, to find the zero of $h(x)$. The iterative method for determining the zeros of $h(x)$ aligns with the well-known third-order iterative technique, Halley's method, which is employed to find the zero of $f(x)$. If $f(x)$ is a non-trivial solution of \eqref{ch5_f_ode}, it suffices to evaluate the functions $r(x)$ and $f(x)$ in order to determine the second derivative of $f(x)$ for Halley's method. To further reduce computational cost, a modification to Halley's method is proposed. In the modified Halley's method, the function $r(x)$ is evaluated precisely once at an appropriate point. In other words, beyond the first iteration, the modified Halley's method requires only the evaluation of the function $f(x)$ and its first derivative, like Newton's method. It is worth mentioning that the modified version of Halley's method retains the third-order convergence. However, in contrast to Newton's method, Halley's method and its modified version have better non-local convergence properties between successive singular points of $h(x)$, i.e., they can converge more effectively even when starting from points farther from the zero. Similar to the results in \cite{Se10}, using bounds on the distance between the zeros of $f(x)$ and $f'(x)$, this study presents suitable initial guesses for Halley's method and its modification to obtain the specific zero. Based on the proposed modified Halley's method, this paper develops efficient algorithms to determine the nodes and weights for both Gauss-Hermite and Gauss-Legendre quadrature. For polynomials of large degree, the proposed algorithms yield all the nodes and weights more rapidly than those presented in \cite{GST19, GST21a}, which rely on the fourth-order method in \cite{Se10}. It is important to note that for polynomials of large degree, the proposed third-order algorithms require significantly fewer iterations compared to the third-order algorithm presented in \cite{PSV25} when calculating all the nodes for both Gauss-Hermite and Gauss-Legendre quadrature. Although the proposed algorithms are third order, the difference in the total number of iterations compared to fourth-order algorithms in \cite{GST19, GST21a} is relatively minor. In contrast to the third-order algorithm in \cite{PSV25} for finding zeros of Legendre polynomials, the proposed algorithm based on the modified Halley's method achieves higher accuracy. This accuracy is comparable to that of the fourth-order algorithm in \cite{GST21a}.

The key contributions of this paper are as follows:
\begin{itemize} 
	\item A computationally efficient third-order modified Halley's method has been presented for finding the zeros of the solution to second-order ODEs.
	\item Based on the behaviour of $r(x)$, nonlocal convergence results for Halley's method and modified Halley's method have been obtained. 
	\item Algorithms are developed based on modified Halley's method to compute all nodes and weights for the Gauss-Hermite and Gauss-Legendre quadratures.
	\item The performance of the proposed algorithms is compared with recent algorithms \cite{GST19, GST21a} that utilise a fourth-order method\cite{Se10}.
    \item For polynomials of a large degree, the proposed algorithms generate all nodes and weights more rapidly than the fourth-order algorithms \cite{GST19, GST21a}.
\end{itemize}   
The rest of the paper is structured as follows: Preliminary results from the literature are presented in Section 2. Construction and non-local convergence of the proposed third-order method are discussed in Section 3. Section 4 presents algorithms for finding all nodes and weights of the Gauss-Hermite and Gauss-Legendre quadratures. Finally, Section 5 presents the conclusions of this paper.

\section{\bf Preliminaries}
Consider a second-order linear differential equation of the form 
\begin{equation*}
 u''(x)+c_1(x)u'(x)+c_2(x)u(x)=0,   
\end{equation*}
where $c_1(x)$ and $c_2(x)$ are real valued functions. The normal form of this differential equation can be obtained by employing the well-known transformation $f(x)=e^{\frac{1}{2}\int c_1(x) {\rm d}x}u(x)$. It is straightforward to observe that the zeros of $f(x)$ and $u(x)$ coincide. Moreover, $f(x)$ satisfies the second-order ODE \eqref{ch5_f_ode} with the coefficient function $r(x)=c_2(x)-\frac{c_1^2(x)}{4}-\frac{c_1'(x)}{2}$. Define $h(x)=\frac{f(x)}{f'(x)}$, as in \cite{Se10}. Then the zeros of $f(x)$ and $h(x)$ coincide. It is easy to verify that $h(x)$ satisfies the following Riccati ODE:
	\begin{equation}\label{ch5_riccati}
		h'(x)=1+r(x)h^2(x).
	\end{equation}
It is worth noting that if $r(x)>0$, then zeros of $f(x)$ and $f'(x)$ are simple and interlaced. Let $z_{f_1}$, $z_{f_2}$ denote the consecutive zeros of $f(x)$ and $z_{f'_1}$, $z_{f'_2}$ represent the consecutive zeros of $f'(x)$ arranged such that $z_{f'_1}<z_{f_1}<z_{f'_2}<z_{f_2}$. Using the Riccati differential equation, interesting bounds between the zero of $f(x)$ and the zero of $f'(x)$ are derived in \cite{Se10}. The significance of these bounds lies in their ability to offer an initial estimate for locating the zeros of $h(x)$. We will conclude this section by presenting the lemmas that establish the bounds, along with an interesting remark concerning the sign of $h(x)$ in the interval $\mathcal{I} = (z_{f'_1}, z_{f'_2})$.

\begin{lemma}\label{ch5_lemma0}\cite[Lemma 2.1]{Se10}
Let $z_{f}$ and $z_{f'}$ represent any consecutive zeros of $f(x)$ and $f'(x)$, respectively. If $r(x)$ is positive and satisfies $0<k<r(x)<K$ between $z_{f}$ and $z_{f'}$, then $\frac{\pi}{2\sqrt{K}}<|z_{f'}-z_{f}|<\frac{\pi}{2\sqrt{k}}$.
\end{lemma}

\begin{lemma}\label{ch5_lemma1}\cite{Se10}
   Let $r(x)$ be a positive function in $(z_{f_1}, z_{f_2})$. 
   \begin{enumerate}
    \item If $r'(x)>0$ in $(z_{f_1}, z_{f_2})$, then $\frac{\pi}{2\sqrt{r(z_{f_2})}}< z_{f_2}-z_{f'_2} <\frac{\pi}{2\sqrt{r(z_{f'_2})}}$ and $\frac{\pi}{2\sqrt{r(z_{f'_2})}}< z_{f'_2}-z_{f_1} <\frac{\pi}{2\sqrt{r(z_{f_1})}}$.
	\item If $r'(x)<0$ in $(z_{f_1}, z_{f_2})$, then $\frac{\pi}{2\sqrt{r(z_{f'_2})}}< z_{f_2}-z_{f'_2} <\frac{\pi}{2\sqrt{r(z_{f_2})}}$ and $\frac{\pi}{2\sqrt{r(z_{f_1})}}< z_{f'_2}-z_{f_1} <\frac{\pi}{2\sqrt{r(z_{f'_2})}}$.
	\end{enumerate}
\end{lemma}
The following lemma is an immediate consequence of Lemma \ref{ch5_lemma1}.
\begin{lemma}\label{ch5_lemma2}\cite{Se10}
	Let $r(x)$ be a positive function in $(z_{f_1}, z_{f_2})$.
	\begin{enumerate}
		\item If $r'(x)>0$ in $(z_{f_1}, z_{f_2})$, then $z_{f_1}<z_{f_2}-\frac{\pi}{\sqrt{r(z_{f_2})}}$.
		\item If $r'(x)<0$ in $(z_{f_1}, z_{f_2})$, then $z_{f_2}>z_{f_1}+\frac{\pi}{\sqrt{r(z_{f_1})}}$.
	\end{enumerate}
\end{lemma}	
By using Lemmas \ref{ch5_lemma1} \& \ref{ch5_lemma2}, one can get the following bounds.
\begin{lemma}\label{ch5_lemma3}
	Let $c$ be a positive constant.
	\begin{enumerate}
		\item If $r'(x)>0$ in $(z_{f_1}, z_{f_2})$ and $c\leq r(x)\leq 4c$ in $(z_{f'_2}, z_{f_2})$, then $z_{f_1}<z_{f_2}-\frac{\pi}{\sqrt{r(z_{f_2})}}<z_{f'_2}$.
		\item If $r'(x)<0$ in $(z_{f_1}, z_{f_2})$ and $c\leq r(x)\leq 4c$ in $(z_{f_1}, z_{f'_2})$, then $z_{f'_2}<z_{f_1}+\frac{\pi}{\sqrt{r(z_{f_1})}}<z_{f_2}$.
	\end{enumerate}
\end{lemma}	
\begin{proof}
	By using part \textit{(1)} of Lemma \ref{ch5_lemma1}, we have $z_{f_2}-z_{f'_2}<\frac{\pi}{2\sqrt{r(z_{f'_2})}}\leq \frac{\pi}{2\sqrt{c}}.$ Consequently, 
	\begin{equation}\label{ch5_eq2}
		z_{f_2}-\frac{\pi}{2\sqrt{c}}< z_{f'_2}.
	\end{equation}
	From Lemma \ref{ch5_lemma2} we have, $z_{f_1}<z_{f_2}-\frac{\pi}{\sqrt{r(z_{f_2})}}$. Thus, using \eqref{ch5_eq2}, we conclude that $z_{f_1}<z_{f_2}-\frac{\pi}{\sqrt{r(z_{f_2})}}\leq z_{f_2}-\frac{\pi}{2\sqrt{c}}<z_{f'_2}.$ Proof for part \textit{(2)} follows similarly.
\end{proof}

\begin{remark}\label{ch5_remark1}
	From \eqref{ch5_riccati}, it is evident that $h'(z_{f_1})=1$. As a result, $h(x)<0$ for all $x$ in the interval $(z_{f'_1}, z_{f_1})$ and $h(x)>0$ for all $x$ in the interval $(z_{f_1}, z_{f'_2})$.
\end{remark}

\section{\bf Derivation and nonlocal convergence of modified Halley's method}
This section presents the derivation of the proposed modified Halley's method and studies its convergence analysis. For the rest of the discussion in this section, $f(x)$ denotes the nontrivial solution of the second-order differential equation \eqref{ch5_f_ode} and $h(x)$ denotes the ratio $\frac{f(x)}{f'(x)}$. Consequently, $h(x)$ satisfies the Riccati differential equation \eqref{ch5_riccati}. Let $z_{f_1}$ be the zero of $h(x)$. The integration of the Riccati equation \eqref{ch5_riccati} from $z^*$ to $x$ yields the following expression: 
\begin{equation*}
    h(x) = \int_{z^*}^{x} (1+r(t)h^2(t)){\rm d}t.
\end{equation*}
By approximating the above integral using the trapezoidal rule, one obtains that
\begin{align*}
		h(x) &\approx \frac{1}{2}(x-z^*)(2+r(x)h^2(x))\\
		z^* &\approx x-\frac{2h(x)}{2+r(x)h^2(x)}.
\end{align*}
If $x$ is an approximation for $z^*$, then one can expect that
\begin{equation}\label{ch5_halley}
		g(x)=x-\frac{2h(x)}{2+r(x)h^2(x)}
\end{equation}
will provide a better approximation of $z^*$. Thus, we derive an iterative scheme defined by the equation $x_{n+1} = g(x_n)$, i.e., 
\begin{equation}\label{ch5_scheme}
		x_{n+1}=x_n-\dfrac{2h(x_n)}{2+r(x_n)h^2(x_n)}, \hspace{2mm} n=0, 1, 2, \dots
\end{equation}
for approximating zeros of $h(x)$. It is straightforward to verify that $g(z^*)=z^*$, $g'(z^*)=0$, $g''(z^*)=0$, and $g'''(z^*)\neq 0$. For $n \in \mathbb{N}$, define $e_n=x_n-z^*$. Consequently,
\begin{equation*}
    e_{n+1}=\dfrac{g'''(z^*)}{6}e^3_{n}+\mathcal{O}(e_n^4).
\end{equation*}
Thus, the iterative scheme \eqref{ch5_scheme} exhibits a third-order convergence property. Let $a\neq 0$ be a real constant. Now consider the function 
\begin{equation}\label{modified_scheme}
    G(x) = x - \frac{2h(x)}{2+ah^2(x)}
\end{equation}
and the corresponding iterative scheme $x_{n+1}=G(x_n)$. It is easy to verify that $G(z^*)=z^*$, $G'(z^*)=0$, $G''(z^*)=0$, and $G'''(z^*)\neq 0$. For $n \in \mathbb{N}$, define $e'_n=x_n-z^*$. Consequently, 
\begin{equation*}
  e'_{n+1}=\dfrac{G'''(z^*)}{6}{e'_{n}}^3+\mathcal{O}({e'_n}^4).  
\end{equation*}
Thus, the corresponding iterative scheme also exhibits third-order convergence. For the initial guess $x_0$, the choice $a=r(x_0)$ leads to a computationally economical iterative method.
\begin{equation}\label{mscheme}
		x_{n+1}=x_n-\dfrac{2h(x_n)}{2+r(x_0)h^2(x_n)}, \hspace{2mm} n=0, 1, 2, \cdots.
\end{equation}
The following remark presents the relationship between the proposed iterative schemes \eqref{ch5_scheme} and \eqref{mscheme} and the well-known Halley's method. 
\begin{remark}
	Using the definition of $h(x)$ and the relation $r(x) = -\frac{f''(x)}{f(x)}$ in the iterative method described in \eqref{ch5_scheme} for finding the zero of $h(x)$, it transforms into Halley's method for finding the zero of $f(x)$, i.e.,
\begin{equation*}
    x_{n+1} = x_n - \frac{2f(x_n)f'(x_n)}{2(f'(x_n))^2-f''(x_n)f(x_n)}, \hspace{2mm} n = 0, 1, 2, \dots .
\end{equation*}
Similarly, using the definition of $h(x)$ and the relation $r(x_0) = -\frac{f''(x_0)}{f(x_0)}$ in the iterative method described in \eqref{mscheme} for finding the zero of $h(x)$, it transforms into a iterative method for finding the zero of $f(x)$, i.e.,
\begin{equation*}
    x_{n+1} = x_n - \frac{2f(x_n)f(x_0)f'(x_n)}{2f(x_0)(f'(x_n))^2-f''(x_0)f(x_n)}, \hspace{2mm} n = 0, 1, 2, \dots .
\end{equation*}
It is interesting to note that the above iterative scheme requires the evaluation of the double derivative only once, unlike Halley's method. Beyond the first iteration, it requires the same function evaluations as Newton's method. We call the iterative scheme in \eqref{mscheme} as \textit{modified Halley's method}. 
\end{remark}
\subsection{Convergence Analysis - Modified Halley's Method}
Let $z_{f'_1}< z_{f'_2}$ be two consecutive zeros of $f'(x)$. Let $z_{f_1}$ be the unique zero of $f(x)$ in the interval $\mathcal{I}=(z_{f'_1}, z_{f'_2})$. Throughout this manuscript, $\tilde{G}(x)$ denotes the iteration function $x-\frac{2h(x)}{2+r(x_0)h^2(x)}$ of the modified Halley's method. Hereafter, this section discusses the nonlocal convergence theorem for the proposed modified Halley's method within the interval $\mathcal{I}$. More specifically, this section uses the behaviour of the coefficient function $r(x)$ to derive the nonlocal convergence behaviour of the proposed iterative scheme $\eqref{mscheme}$. Assuming that the function $r(x)$ is monotonic, the following theorem establishes the non-local convergence of the modified iterative scheme \eqref{mscheme} within the interval $\mathcal{I}$.
\begin{theorem}\label{ch5_thm11}
Let $r(x)$ be a positive function in the interval $\mathcal{I}$.
	\begin{enumerate}
		\item Let $r(x)$ be non-increasing in the interval $(z_{f'_1}, z_{f_1})$. Then, the modified Halley's method \eqref{mscheme} exhibits monotonic convergence to $z_{f_1}$ for any initial guess $x_0$ within the interval $(z_{f'_1}, z_{f_1})$.
        \item Let $r(x)$ be non-decreasing in the interval $(z_{f_1}, z_{f'_2})$. Then, the modified Halley's method \eqref{mscheme} exhibits monotonic convergence to $z_{f_1}$ for any initial guess $x_0$ within the interval $(z_{f_1}, z_{f'_2})$.
	\end{enumerate}
\end{theorem}
\begin{proof}
\textit{(1)}. For all $x \in (z_{f'_1}, z_{f_1})$, it is evident that $2+r(x_0)h^2(x)>0$. Hence, the proposed modified Halley's method \eqref{mscheme} is well-defined for $x$ in $(z_{f'_1}, z_{f_1})$. In addition,
\begin{equation}\label{ch5_modified_derivative}
		\tilde{G}'(x) = \dfrac{4h^2(x)(r(x_0)-r(x))+r^2(x_0)h^4(x)+2r(x_0)h^2(x)h'(x)}{(2+r(x_0)h^2(x))^2}, ~\forall~ x \in \mathcal{I}.
\end{equation}
As $r(x)$ is non-increasing, it follows that $\tilde{G}'(x)>0$ for all $x \in (x_0, z_{f_1})$. Clearly $x_1=\tilde{G}(x_0)$ exists. Using Remark \ref{ch5_remark1}, it can be concluded that $x_0<x_1$. Note that
\begin{equation*}
z_{f_1}-x_1=\tilde{G}(z_{f_1})-\tilde{G}(x_0)=\tilde{G}'(c)(z_{f_1}-x_0)>0, \quad \text{for some }\quad c \in (x_0, z_{f_1}).
\end{equation*}
Hence, $x_1<z_{f_1}$. Thus, $x_0<x_1<z_{f_1}$. Assume that $x_2, x_3, \cdots, x_n$ exist and satisfying the relations $x_0<x_1 < x_2 < x_3 < \cdots < x_n < z_{f_1}$. Clearly, $x_{n+1}=\tilde{G}(x_n)$ exists. Remark \ref{ch5_remark1} ensures that $x_n<x_{n+1}$. Furthermore,  
\begin{equation*}
  z_{f_1}-x_{n+1}=\tilde{G}(z_{f_1})-\tilde{G}(x_n)=\tilde{G}'(c)(z_{f_1}-x_n)>0, \quad\text{for some }\quad c \in (x_{n}, z_{f_1}).  
\end{equation*}
Hence, $x_{n+1}<z_{f_1}$. Thus, the sequence $(x_n)$ obtained by the scheme \eqref{mscheme} is monotonically increasing and bounded above by $z_{f_1}$. Hence, the sequence $(x_n)$ generated by the modified Halley's method \eqref{mscheme} converges. One can verify that the limit of $x_n$ is $z_{f_1}$. The proof for the second part is left over because it is similar.
\end{proof}
The above theorem guarantees the convergence of modified Halley's method either in $(z_{f'_1}, z_{f_1})$ or in $(z_{f_1}, z_{f'_2})$. The subsequent theorem ensures the convergence of modified Halley's method in the interval $\mathcal{I}=(z_{f'_1}, z_{f'_2})$ through a suitable alteration of the scheme. 
\begin{theorem}\label{coupled_theorem_1}
Let $r(x)$ be a positive and non-increasing function in the interval $\mathcal{I}$. Let $x_0,y_0 \in \mathcal{I}$ such that $x_0<z_{f_1}<y_0$. Let $(x_n)$ and $(y_n)$ be sequences generated as follows:
\begin{eqnarray}
\label{coupled_1}
\begin{cases}
x_{n+1} = x_n - \frac{2h(x_n)}{2 + r(x_0)h^2(x_n)},\\
y_{n+1} = y_n - \frac{2h(y_n)}{2 + r(x_0)h^2(y_n)}.
\end{cases}
\end{eqnarray}
Then the sequences $(x_n)$ and $(y_n)$ converge monotonically to the unique zero $z_{f_1}$ and satisfy the following relation:
\begin{equation}\label{ch5_relation_1}
        x_0 < x_1 < x_2 < \cdots < x_n < z_{f_1} < y_n < \cdots < y_2 < y_1 < y_0.
\end{equation}    
\end{theorem}
\begin{proof}
The monotone convergence of the sequence $(x_n)$ follows directly from Theorem \ref{ch5_thm11}. The function $\tilde{G}(y) = y-\frac{2h(y)}{2+r(x_0)h^2(y)}$ is well defined in the interval $(z_{f_1},z_{f'_2})$ and $\tilde{G}'(y)>0$ in the interval $(z_{f_1},y_0)$. Clearly $y_1=\tilde{G}(y_0)$ exists. Using Remark \ref{ch5_remark1}, it can be concluded that $y_1<y_0$. Note that
\begin{equation*}
    y_1-z_{f_1} = \tilde{G}(y_0)-\tilde{G}(z_{f_1}) = \tilde{G}'(c)(y_0-z_{f_1})>0~~~\text{for some}~~c \in (z_{f_1}, y_0).
\end{equation*}
Thus, $z_{f_1}<y_1$. Hence, $z_{f_1}<y_1<y_0$. Assume that $y_2, y_3, \cdots, y_n$ exist and satisfying the relations $y_0>y_1 > y_2 > y_3 > \cdots > y_n > z_{f_1}$. Clearly, $y_{n+1}=\tilde{G}(y_n)$ exists. Remark \ref{ch5_remark1} ensures that $y_n>y_{n+1}$. Furthermore,  
\begin{equation*}
    y_{n+1}-z_{f_1} = \tilde{G}(y_{n})-\tilde{G}(z_{f_1}) = \tilde{G}'(c)(y_{n}-z_{f_1})>0~~~\text{for some}~~c \in (z_{f_1}, y_{n}).
\end{equation*}
Hence, $y_{n+1}>z_{f_1}$. Thus, the sequence $(y_n)$ obtained by the scheme \eqref{coupled_1} is monotonically decreasing and bounded below by $z_{f_1}$. Hence, the sequence $(y_n)$ generated by the modified Halley's method \eqref{coupled_1} converges. One can verify that the limit of $y_n$ is $z_{f_1}$.     
\end{proof}
\begin{theorem}
Let $r(x)$ be a positive and non-decreasing function in the interval $\mathcal{I}$. Let $x_0,y_0 \in \mathcal{I}$ such that $x_0<z_{f_1}<y_0$. Let $(x_n)$ and $(y_n)$ be sequences generated as follows:
\begin{eqnarray}
\label{coupled_2}
\begin{cases}
x_{n+1} = x_n - \frac{2h(x_n)}{2 + r(y_0)h^2(x_n)},\\
y_{n+1} = y_n - \frac{2h(y_n)}{2 + r(y_0)h^2(y_n)}.
\end{cases}
\end{eqnarray}
Then the sequences $(x_n)$ and $(y_n)$ converge monotonically to the unique zero $z_{f_1}$ and satisfy the relation \eqref{ch5_relation_1}.
\end{theorem}
\begin{proof}
    The proof is similar to the proof of Theorem \ref{coupled_theorem_1}, and hence it is omitted.
\end{proof}
The theorems presented above establish nonlocal convergence results for the proposed modified Halley's method, assuming that $r(x)$ is a monotone function. The following theorem provides sufficient conditions for relaxing the monotonicity of the function $r(x)$ and ensures the nonlocal convergence property for the modified Halley's method in the entire interval $\mathcal{I}=(z_{f'_1},z_{f'_2})$.
\begin{theorem}\label{ch5_thm22}
	Let $c_1$ and $c_2$ be positive constants such that $c_2 < r(x) < c_1$ in the interval $\mathcal{I}$ and $\frac{c_1}{c_2}\leq \frac{3}{2}$. The proposed modified Halley's method converges monotonically to the solution $z_{f_1}$ for any initial guess $x_0$ in the interval $\mathcal{I}$.
\end{theorem}
\begin{proof}
Clearly, the function $\tilde{G}(x)$ is well-defined in the interval $\mathcal{I}$. Using the fact that $h'(x)\geq 1$ and $-r(x)>-c_1$ for all $x\in \mathcal{I}$ in Eq.\eqref{ch5_modified_derivative}, one can get
\begin{equation*}
 \tilde{G}'(x) > \dfrac{h^2(x)(c_2^2h^2(x)+6c_2-4c_1)}{(2+r(x_0)h^2(x))^2} \geq 0 \quad \forall \hspace{2mm} x \in \mathcal{I}.   
\end{equation*}
Let $x_0 \in (z_{f'_1},z_{f_1})$. Hence, $x_1=\tilde{G}(x_0)$ exists. Using Remark \ref{ch5_remark1}, it can be concluded that $x_0<x_1$. Note that
\begin{equation*}
	z_{f_1} - x_1 = \tilde{G}(z_{f_1}) - \tilde{G}(x_0) = \tilde{G}'(c) (z_{f_1} - x_0) > 0 \hspace{2mm}\text{for some} \hspace{2mm} c \in (x_0, z_{f_1}).
\end{equation*}
Hence, $x_1<z_{f_1}$. Thus, $x_0<x_1<z_{f_1}$. Assume that $x_2, x_3, \cdots, x_n$ exist and satisfying the relations $x_0<x_1 < x_2 < x_3 < \cdots < x_n < z_{f_1}$. Clearly, $x_{n+1}=\tilde{G}(x_n)$ exists. Remark \ref{ch5_remark1} ensures that $x_n<x_{n+1}$. Furthermore,  
\begin{equation*}
  z_{f_1}-x_{n+1}=\tilde{G}(z_{f_1})-\tilde{G}(x_n)=\tilde{G}'(c)(z_{f_1}-x_n)>0, \quad\text{for some }\quad c \in (x_{n}, z_{f_1}).  
\end{equation*}
Hence, $x_{n+1}<z_{f_1}$. Thus, the sequence $(x_n)$ obtained by the scheme \eqref{mscheme} is monotonically increasing and bounded above by $z_{f_1}$. Hence, the sequence $(x_n)$ generated by the modified Halley's method \eqref{mscheme} converges. One can verify that the limit of $x_n$ is $z_{f_1}$. The proof is similar if $x_0\in (z_{f_1},z_{f'_2})$ and hence it is omitted.
\end{proof}
\subsection{Convergence Analysis - Halley's Method}
Let $z_{f'_1}< z_{f'_2}$ be two consecutive zeros of $f'(x)$. Let $z_{f_1}$ be the unique zero of $f(x)$ in the interval $\mathcal{I}=(z_{f'_1}, z_{f'_2})$. This subsection presents the nonlocal convergence results for Halley's method in the interval $\mathcal{I}$, based on the behaviour of the coefficient function. Throughout this manuscript, $g(x)$ denotes the iteration function $x-\frac{2h(x)}{2+r(x)h^2(x)}$ of Halley's method. Similar to Theorem \ref{ch5_thm11}, the following theorem presents the convergence result for Halley’s method \eqref{ch5_scheme}. 
\begin{theorem}\label{theorem_halley_1}
	Let $r(x)$ be a positive function in the interval $\mathcal{I}$.
	\begin{enumerate}
		\item  Let $r(x)$ be non-increasing in the interval $(z_{f'_1}, z_{f_1})$. Then, Halley's method \eqref{ch5_scheme} exhibits monotonic convergence to $z_{f_1}$ for any initial guess $x_0$ within the interval $(z_{f'_1}, z_{f_1})$.
	    \item Let $r(x)$ be non-decreasing in the interval $(z_{f_1}, z_{f'_2})$. Then, Halley's method \eqref{ch5_scheme} exhibits monotonic convergence to $z_{f_1}$ for any initial guess $x_0$ within the interval $(z_{f_1}, z_{f'_2})$.
	\end{enumerate}
\end{theorem}
\begin{proof}
\textit{(1)}. For all $x \in (z_{f'_1}, z_{f_1})$, it is evident that $2+r(x)h^2(x)>0$. Hence, Halley's method \eqref{ch5_scheme} is well-defined for $x$ in $(z_{f'_1}, z_{f_1})$. In addition,
 	\begin{equation}\label{ch5_g_derivative}
 		g'(x)=\dfrac{3 r^2(x)h^4(x)+2 r'(x)h^3(x)+2 r(x)h^2(x)}{(2+r(x)h^2(x))^2}, ~\forall~ x \in \mathcal{I}.
 	\end{equation}
As $r(x)$ is non-increasing, it follows that $g'(x)>0$ for all $x \in (z_{f'_1}, z_{f_1})$. Clearly $x_1=g(x_0)$ exists. Using Remark \ref{ch5_remark1}, it can be concluded that $x_0<x_1$. Note that
\begin{equation*}
z_{f_1}-x_1=g(z_{f_1})-g(x_0)=g'(c)(z_{f_1}-x_0)>0, \quad \text{for some }\quad c \in (x_0, z_{f_1}).
\end{equation*}
Hence, $x_1<z_{f_1}$. Thus, $x_0<x_1<z_{f_1}$. Assume that $x_2, x_3, \cdots, x_n$ exist and satisfying the relations $x_0<x_1 < x_2 < x_3 < \cdots < x_n < z_{f_1}$. Clearly, $x_{n+1}=g(x_n)$ exists. Remark \ref{ch5_remark1} ensures that $x_n<x_{n+1}$. Furthermore,  
\begin{equation*}
  z_{f_1}-x_{n+1}=g(z_{f_1})-g(x_n)=g'(c)(z_{f_1}-x_n)>0, \quad\text{for some }\quad c \in (x_{n}, z_{f_1}).  
\end{equation*}
Hence, $x_{n+1}<z_{f_1}$. Thus, the sequence $(x_n)$ obtained by the scheme \eqref{ch5_scheme} is monotonically increasing and bounded above by $z_{f_1}$. Hence, the sequence $(x_n)$ generated by Halley's method \eqref{ch5_scheme} converges. One can verify that the limit of $x_n$ is $z_{f_1}$. The proof for the second part is left over because it is similar.
\end{proof}
Assuming that $r(x)$ is monotonic, Theorem \ref{theorem_halley_1} shows that Halley's method \eqref{ch5_scheme} converges when the initial guess is chosen only from a sub-interval of $\mathcal{I}$ where $h(x)$ and $r'(x)$ have the same sign. The following theorem establishes further sufficient conditions that guarantee the nonlocal convergence of Halley's method \eqref{ch5_scheme} within a subinterval of $\mathcal{I}$ where the functions $r'(x)$ and $h(x)$ exhibit opposite signs.
\begin{theorem}\label{theorem_halley_2}
	Let $k_1$ and $k_2$ be two positive constants such that $\frac{k_2^3}{k_1^2}>\frac{1}{6}$ and $r(x)>k_2$ in the interval $\mathcal{I}$.
	\begin{enumerate}
		\item Let $0\leq r'(x)<k_1$ in the interval $(z_{f'_1}, z_{f_1})$. Then, Halley's method \eqref{ch5_scheme} exhibits monotonic convergence to $z_{f_1}$ for any initial guess $x_0$ within the interval $(z_{f'_1}, z_{f_1})$.
		\item Let $-k_1<r'(x)\leq 0$ in the interval $(z_{f_1}, z_{f'_2})$. Then, Halley's method \eqref{ch5_scheme} exhibits monotonic convergence to $z_{f_1}$ for any initial guess $x_0$ within the interval $(z_{f_1}, z_{f'_2})$.
	\end{enumerate}	
\end{theorem}
\begin{proof}
	\textit{(1)}. Clearly the function $g(x)$ is well-defined in the interval $(z_{f'_1}, z_{f_1})$. Consider the quadratic polynomial 
$P(z)=3k_2^2z^2+2k_1z+2k_2$. One can verify that $P(z)>0$ for all $z\in \mathbb{R}$. Using the fact that $0\leq r'(x)<k_1$ and $r(x)>k_2$ for all $x \in (z_{f'_1}, z_{f_1}) \in \mathcal{I}$ in Eq. \eqref{ch5_g_derivative}, one can get
\begin{equation*}
        g'(x)> \dfrac{h^2(x)P(h(x))}{(2+r(x)h^2(x))^2}>0.
\end{equation*}
Note that $x_1 = g(x_0)$ exists. Using Remark \ref{ch5_remark1}, it can be concluded that $x_0<x_1$. Observe that
   \begin{equation*}
       z_{f_1}-x_1=g(z_{f_1})-g(x_0)=g'(c)(z_{f_1}-x_0)>0, \quad\text{for some }\quad c \in (x_0, z_{f_1}).
   \end{equation*}
Hence, $x_1<z_{f_1}$. Thus, $x_0<x_1<z_{f_1}$. Assume that $x_2, x_3, \cdots, x_n$ exist and satisfying the relations $x_0<x_1 < x_2 < x_3 < \cdots < x_n < z_{f_1}$. Clearly, $x_{n+1}=g(x_n)$ exists. Remark \ref{ch5_remark1} ensures that $x_n<x_{n+1}$. Furthermore,  
\begin{equation*}
  z_{f_1}-x_{n+1}=g(z_{f_1})-g(x_n)=g'(c)(z_{f_1}-x_n)>0, \quad\text{for some }\quad c \in (x_{n}, z_{f_1}).  
\end{equation*}
Hence, $x_{n+1}<z_{f_1}$. Thus, the sequence $(x_n)$ obtained by the scheme \eqref{ch5_scheme} is monotonically increasing and bounded above by $z_{f_1}$. Hence, the sequence $(x_n)$ generated by the modified Halley's method \eqref{mscheme} converges. One can verify that the limit of $x_n$ is $z_{f_1}$. One can prove the second part in a similar way.
\end{proof}
By clubbing Theorem \ref{theorem_halley_1} and Theorem \ref{theorem_halley_2}, one can obtain the following version of the global convergence theorem in the interval $\mathcal{I}$.
\begin{theorem}\label{theorem_halley_3}
    Let $r(x)$ be a positive and monotone function on $\mathcal{I}$. Let $k_1$ and $k_2$ be positive constants such that $r(x)>k_2$, $|r'(x)|<k_1$ in the interval $\mathcal{I}$ and $\frac{k_2^3}{k_1^2}>\frac{1}{6}$. Then, Halley's method \eqref{ch5_scheme} converges monotonically to the solution $z_{f_1}$ for any initial guess $x_0$ in the interval $\mathcal{I}$.	
\end{theorem}
We will conclude our discussion of this subsection by emphasising an advantage of the proposed modified Halley's method when compared to both the original Halley's method and the fourth-order method presented in \cite{Se10}. 
\begin{remark}
All the available nonlocal convergence theorems for the fourth-order method \cite{Se10} and Halley's method in this manuscript necessitate the monotonicity of $r(x)$. However, the nonlocal convergence result in Theorem \ref{ch5_thm22} for the modified Halley's method does not require the monotonicity of the coefficient function $r(x)$. Consider the following second-order differential equation:
\begin{equation}
\label{example_ode}
  y''+\left(\frac{5}{4}+\frac{1}{4}\sin(4x^2)\right)y=0.  
\end{equation} 
Let $f(x)$ be a non-trivial solution of the differential equation \eqref{example_ode} and define $r(x)=\frac{5}{4}+\frac{1}{4}\sin(4x^2)$. It is evident that $1 \leq r(x) \leq \frac{3}{2}$ for all $x \in \mathbb{R}$ and $r'(x)=2x\cos(4x^2)$. Consequently, $r'(x)$ changes signs in any interval whose length exceeds $\frac{\sqrt{\pi}}{2}$. Theorem \cite[Theorem 31.7]{AO08} ensures that $f(x)$ has infinitely many zeros in $(0, \infty)$. Let $z_{f'_1}<z_{f'_2}$ be two consecutive positive zeros of $f'(x)$ and $z_{f_1}\in (z_{f'_1}, z_{f'_2})$ be the unique zero of $f(x)$. Lemma \ref{ch5_lemma0} guarantees that
\begin{equation*}
   z_{f'_2}-z_{f_1}>\frac{\pi}{\sqrt{6}}~~~\text{and}~~~ z_{f_1}-z_{f'_1}>\frac{\pi}{\sqrt{6}}.
\end{equation*}
As $\frac{\sqrt{\pi}}{2}<\frac{\pi}{\sqrt{6}}$, $r'(x)$ changes signs in both the intervals $(z_{f'_1}, z_{f_1})$ and $(z_{f_1}, z_{f'_2})$. In other words $r$ is not a monotone function neither in $(z_{f'_1}, z_{f_1})$ nor in $(z_{f_1}, z_{f'_2})$. The nonlocal convergence results available in \cite{Se10} and in this manuscript, corresponding to the fourth-order method and Halley's method, respectively, are not usable for the intervals $(z_{f'_1}, z_{f_1})$ and $(z_{f_1}, z_{f'_2})$. Theorem \ref{ch5_thm22} confirms that the proposed modified Halley's method exhibits nonlocal convergence throughout the entire interval $\mathcal{I}=(z_{f'_1}, z_{f'_2})$ for this problem.     
\end{remark}
\subsection{Initial guesses for multiple zeros}
In this subsection, initial guesses are provided for the proposed modified Halley's method to approximate all zeros of $f(x)$ within a specific interval, utilising the bounds from Lemma \ref{ch5_lemma0} to Lemma \ref{ch5_lemma3}. This discussion parallels that found in \cite{Se10} regarding the fourth-order method. Let $z_i$ and $z'_i$ denote the $i^{\text{th}}$ zeros of the function $f(x)$ and its derivative $f'(x)$, respectively. Let $n\in \mathbb{N}$. Assume that the interval $\mathcal{J}'=[z'_i,z'_{i+n+1}]$ contains $n+1$ zeros of the function $f(x)$, say $z_k \in \mathcal{J}'$, $k=i, i+1,\cdots, i+n$. In other words,  
\begin{equation*}
z'_i<z_i<z'_{i+1}<z_{i+1}<\cdots<z'_{i+n}<z_{i+n}<z'_{i+n+1}.
\end{equation*}
Let $z_0^{k}$ denote the suitable initial guess for the proposed modified Halley's method \eqref{mscheme} to approximate the zero $z_k$ of $f(x)$. Let $c>0$. If $r(x)$ is a monotonic function and it satisfies the condition $c \leq r(x) \leq 4c$ in the interval $\mathcal{J}'$, then Table \ref{table_J'_c<r<4c} provides solid initial guesses for the modified Halley's method to compute all the zeros in the interval $\mathcal{J}'$. Theorem \ref{ch5_thm11} ensures the convergence of modified Halley's method for the initial guesses presented in Table \ref{table_J'_c<r<4c}. 

\begin{table}[ht]
\centering
\small
\renewcommand{\arraystretch}{1.0}
\setlength{\tabcolsep}{6pt}

\begin{tabular}{|c|p{6.5cm}|c|}
\hline
\multicolumn{3}{|c|}{\textbf{Let $\boldsymbol{c>0}$, \; $\boldsymbol{c\leq r(x) \leq 4c}$ in $\boldsymbol{\mathcal{J'}}$}} \\ \hline

\textbf{Sign of $\boldsymbol{r'(x)}$} 
& \textbf{Initial guesses to find all zeros in $\boldsymbol{\mathcal{J'}}$}
& \textbf{Justification} \\ \hline

$r'(x)>0$ 
& $z_0^{i+n}=z'_{i+n+1}-\frac{\pi}{2\sqrt{r(z'_{i+n+1})}} \in (z_{i+n}, z'_{i+n+1})$, 
& Lemma~\ref{ch5_lemma1}\\
& for $j=i+n, \cdots, i+2, i+1:$\hspace{2mm}
& \hspace{1mm}\\
& $z_0^{j-1} = z_{j}-\frac{\pi}{\sqrt{r(z_{j})}} \in (z_{j-1}, z'_{j})$
& Lemma~\ref{ch5_lemma3}\\ \hline

$r'(x)<0$ 
& $z_0^{i}=z'_{i}+\frac{\pi}{2\sqrt{r(z'_{i})}} \in (z'_i, z_i)$,
& Lemma~\ref{ch5_lemma1}\\
& for $k=i, i+2, \cdots i+n-1$: \hspace{3mm}
& \hspace{1mm}\\
& $z_0^{k+1}=z_{k}+\frac{\pi}{\sqrt{r(z_{k})}} \in (z'_{k+1},z_{k+1})$
& Lemma~\ref{ch5_lemma3}\\ \hline

\end{tabular}

\caption{Initial guesses for computing zeros in $\mathcal{J'}$ if $r$ is monotone and $c \leq r(x) \leq 4c$}
\label{table_J'_c<r<4c}

\end{table}
The conditions of monotonicity and the existence of a positive constant $c$ for the coefficient function $r(x)$ restrict the application of the proposed method to a smaller class of problems. This situation resembles the challenges encountered by the fourth-order method discussed in \cite{Se10}. If the coefficient function $r(x)$ is monotone, then a suitable shift function $S_j(x)$ can be employed to derive an appropriate initial guess for the proposed method. Let $r(x)$ be a monotone function within the interval $\mathcal{J}'$. The shift function is defined as follows: 
\begin{equation}
    S_j(x) = x - \frac{j\pi}{2\sqrt{r(x)}}, ~\text{where}~ j = \text{sign}(r'(x)) ~\text{in}~ \mathcal{J}'.
\end{equation}
The shift function mentioned above differs from the one presented in \cite{Se10}. Lemma \ref{ch5_lemma2} ensures that if $r'(z)<0$ in $\mathcal{J}'$, then the inequality $z_k+\frac{\pi}{\sqrt{r(z_k)}} < z_{k+1}$ holds for $k=i,i+1,\cdots,i+n-1$. This inequality is independent of the condition $c \leq r(x) \leq 4c$. However, there is no guarantee that $z_k + \frac{\pi}{\sqrt{r(z_k)}} > z'_{k+1}$ holds when the condition $c \leq r(x) \leq 4c$ is relaxed. If $z_k + \frac{\pi}{\sqrt{r(z_k)}}$ falls within the interval $(z_k, z'_{k+1})$, there is no theoretical guarantee that the proposed method will converge to the next zero $z_{k+1}$ for the initial guess $z_k + \frac{\pi}{\sqrt{r(z_k)}}$. In this case, using the shift function $S_{-1}(x)$ and $z_k+\frac{\pi}{\sqrt{r(z_k)}}$, one can obtain a suitable initial guess to compute $z_{k+1}$. Similarly, if $r(x)$ satisfies the condition $r'(x) > 0$ alone, then $z_j - \frac{\pi}{\sqrt{r(z_j)}}$ may not serve as a suitable initial guess for computing the next zero $z_{j-1}$, where $j = i+n, i+n-1, \dots, i+1$. In this case, using the shift function $S_{1}(x)$ and $z_j-\frac{\pi}{\sqrt{r(z_j)}}$, one can obtain a suitable initial guess to compute $z_{j-1}$. The following lemma is an important tool in this direction.
\begin{lemma}\label{ch5_lemma4} Let $r(x)$ be a positive function in the interval $\mathcal{J'}$.
\begin{enumerate}
  \item  Let $r'(x)<0$ in the interval $\mathcal{J'}$. Then for each $z_k$, $k=i,i+1,\cdots,i+n-1$, there exists a non-negative integer $N_k$ such that $S^{N_k}_{-1}(z_k+\frac{\pi}{\sqrt{r(z_k)}}) \in (z'_{k+1},z_{k+1})$.
  \item    Let $r'(x)>0$ in the interval $\mathcal{J'}$. Then for each $z_{j}$, $j=i+n, i+n-1, \cdots, i+1$, there exists a non-negative integer $N_j$ such that $S^{N_j}_{1}(z_j-\frac{\pi}{\sqrt{r(z_j)}}) \in (z_{j-1},z'_{j})$.
\end{enumerate}
\end{lemma}
\begin{proof}
\textit{(1)}. If $z_{k}+\frac{\pi}{\sqrt{r(z_{k})}} > z'_{k+1}$, then $N_k=0$. Assume that $z_{k}+\frac{\pi}{\sqrt{r(z_{k})}} < z'_{k+1}$. It is straightforward to confirm that $S_{-1}(x)$ is an increasing function on $\mathcal{J}'$. Part 2 of Lemma \ref{ch5_lemma1} ensures that $z'_{k+1}+\frac{\pi}{2\sqrt{r(z'_{k+1})}} < z_{k+1}$. In other words 
\begin{equation}\label{upper_bound_S}
    S_{-1}(z'_{k+1})<z_{k+1}.
\end{equation}
Now consider the successive iterative scheme $x_{n+1}=S_{-1}(x_n)$ with $x_0=z_{k}+\frac{\pi}{\sqrt{r(z_{k})}}$. Hence, $x_n=S^n_{-1}(x_0)$ for all $n\in \mathbb{N}$. Clearly, $(x_n)$ is an increasing sequence. Hence, there exists an integer $N_k$ such that $S^{N_k-1}_{-1}(x_0) \in (z_k,z'_{k+1}]$ and $S^{N_k}_{-1}(x_0) > z'_{k+1}$. In other words,
\begin{eqnarray}
    S^{N_k-1}_{-1}(x_0) \leq z'_{k+1}~~\text{and}~~ z'_{k+1} < S^{N_k}_{-1}(x_0).
\end{eqnarray}
As $S_{-1}(x)$ is an increasing function, the inequality $S^{N_k-1}_{-1}(x_0) \leq z'_{k+1}$ and Eq. \eqref{upper_bound_S} leads to $S_{N_k}(x_0)< z_{k+1}$. Hence, the result follows. One can prove the second part in a similar way.
\end{proof}

\begin{remark}\label{remark_initial_guess*}
Let $r(x)$ be a positive function defined on the interval $\mathcal{J}'$. Let $N_k$ denote the non-negative integer as specified in the proof of Lemma \ref{ch5_lemma4}. Let $\tilde{z}^{k+1}_0$ denote $S_{-1}^{N_k}(z_k + \frac{\pi}{\sqrt{r(z_k)}})$ for $k = i, i+1, \cdots, i+n-1$. If $r'(x)<0$ in $\mathcal{J}'$, then $\tilde{z}^{k+1}_0 \in (z'_{k+1},z_{k+1})$. The conclusion drawn from Theorem \ref{ch5_thm11} guarantees that the proposed method converges to the subsequent zero $z_{k+1}$, using the initial guess $\tilde{z}^{k+1}_0$. Note that if $N_k \in \mathbb{N}$, then $N_k$ is the smallest integer such that $h(S_{-1}^{N_k}(z_k + \frac{\pi}{\sqrt{r(z_k)}}))<0$ and $h(S_{-1}^{N_k-1}(z_k + \frac{\pi}{\sqrt{r(z_k)}}))>0$. Let $\tilde{z}^{j-1}_0$ denote $S_{1}^{N_j}(z_{j} - \frac{\pi}{\sqrt{r(z_j)}})$ for $j = i+n, i+n-1, \cdots, i+1$. If $r'(x)>0$ in $\mathcal{J}'$, then $\tilde{z}^{j-1}_0 \in (z_{j-1},z'_{j})$. Similarly, Theorem \ref{ch5_thm11} guarantees that the proposed method converges to the subsequent zero $z_{j-1}$, using the initial guess $\tilde{z}^{j-1}_0$. Note that if $N_j \in \mathbb{N}$, then $N_j$ is the smallest integer such that $h(S_{-1}^{N_j}(z_j - \frac{\pi}{\sqrt{r(z_j)}}))>0$ and $h(S_{-1}^{N_j-1}(z_j - \frac{\pi}{\sqrt{r(z_j)}}))<0$.
\end{remark}
We conclude the discussion of this section by summarising the results in the Lemma \ref{ch5_lemma4} and the Remark \ref{remark_initial_guess*} as Table \ref{initial_guess_J'_1}.

\begin{table}[ht]
\centering
\footnotesize
\renewcommand{\arraystretch}{1.0}
\setlength{\tabcolsep}{3pt}

\begin{tabular}{|c|p{6.7cm}|p{2.0cm}|}
\hline
\multicolumn{3}{|c|}{\textbf{Let $\boldsymbol{r(x)}$ be positive in $\boldsymbol{\mathcal{J'}}$}} \\ \hline

\textbf{Sign of $\boldsymbol{r'(x)}$} 
& \centering \textbf{Initial guesses to compute all zeros in $\boldsymbol{\mathcal{J'}}$}
& \centering \textbf{Justification}
\arraybackslash \\ \hline

$r'(x)>0$ 
& $z_0^{i+n}= z'_{i+n+1}
-\dfrac{\pi}{2\sqrt{r(z'_{i+n+1})}}
\in (z_{i+n}, z'_{i+n+1})$
& Lemma~\ref{ch5_lemma1} \\

& for $j=i+n,\cdots,i+2,i+1:$ 
& \\

& 
$z_0^{j-1}=\tilde{z}_0^{{j-1}}
\in (z_{j-1},z'_{j})$
& Remark~\ref{remark_initial_guess*} \\ \hline

$r'(x)<0$ 
& $z_0^{i}= z'_{i}
+\dfrac{\pi}{2\sqrt{r(z'_{i})}}
\in (z'_i,z_i)$
& Lemma~\ref{ch5_lemma1} \\

& for $k=i,i+1,\cdots,i+n-1:$ 
& \\

&
$z_0^{k+1}=\tilde{z}_0^{{k+1}}
\in (z'_{k+1},z_{k+1})$
& Remark~\ref{remark_initial_guess*} \\ \hline

\end{tabular}

\caption{Initial guesses for computing zeros in $\mathcal{J'}$ if $r$ is monotone}
\label{initial_guess_J'_1}
\end{table}

\section{Applications}
This section introduces algorithms based on the modified Halley's method \eqref{mscheme} for computing the nodes and weights of Gauss-Legendre and Gauss-Hermite quadrature. This section also compares their performance against methods presented in recent literature. Throughout this section, the algorithms referred to as \textbf{MHM-H} and \textbf{MHM-L} utilise the modified Halley's method for Gauss-Hermite quadrature and Gauss-Legendre quadrature, respectively. The algorithms referred to as \textbf{FOM-L} \cite{GST21a} and \textbf{FOM-H} \cite{GST19} are based on the fourth-order iterative scheme studied in \cite{Se10} for Gauss-Legendre and Gauss-Hermite quadrature, respectively. The algorithms denoted as \textbf{TOM-L} and \textbf{TOM-H} utilise the third-order iterative method proposed in \cite{PSV25} for Gauss-Legendre and Gauss-Hermite quadrature, respectively. All numerical experiments were carried out using MATLAB R2024b (64-bit) on a 64-bit Windows PC equipped with a 12th-generation Intel Core i7 processor and 16 GB of RAM. Let $x_k^{(d)}$ and $x_k^{(q)}$ represent the numerical approximation of the $k^{\text{th}}$ root of a polynomial, as obtained through double-precision arithmetic and extended-precision arithmetic, respectively. Let $w_k^{(d)}$ and $w_k^{(q)}$ represent the numerical approximation of the $k^{\text{th}}$ weight of Gauss quadrature, as obtained through double-precision arithmetic and extended-precision arithmetic, respectively. The relative error (RE) for the $k^{\text{th}}$ zero and the $k^{\text{th}}$ weight is defined by
\begin{equation}\label{relative_error_ch5}
	\text{RE}_{x_k} = \left|1-\frac{x_k^{(d)}}{x_k^{(q)}}\right|; ~~~~~~~~\text{RE}_{w_k} = \left|1-\frac{w_k^{(d)}}{w_k^{(q)}}\right|.
\end{equation}
To compare the performance of the proposed algorithms, which are based on modified Halley's method \eqref{mscheme}, with those derived from the third-order method \cite{PSV25} and the fourth-order method \cite{GST19, GST21a}, we use the maximum relative error concerning the fourth-order method, referred to as MR.Er.F. For a polynomial of degree $n$, MR.Er.F. is defined by 
    \begin{equation} \label{MR.Er-F_ch5}
      \mbox{MR.Er-F} =  \max_{ k=1,2,\cdots, n} \mbox{(Er-F)}_k~, ~~\mbox{where}~~ \mbox{(Er-F)}_k=\left|1 - \frac{x_k^{c}}{x_k^{f}}\right|,
    \end{equation}
$x_k^{f}$ and $x_k^{c}$ represent the $k^{\mbox{th}}$ root of the polynomial calculated using the fourth-order algorithm and the algorithm being compared to the fourth-order algorithm, respectively, through double-precision calculations. Throughout this numerical simulation, the third-order method TOM-H and TOM-L use the robust initial guess recommended in \cite{PSV25}. The FOM-H \cite{GST19}, FOM-L \cite{GST21a}, and the proposed MHM-H and MHM-L use the same initial guesses. The numerical implementations of MHM-H and MHM-L closely resemble those of FOM-H and FOM-L, as described in \footnote{https://github.com/NumericalQuadrature/GaussQuadrature}. The codes used in the comparative study are available in \footnote{\url{https://github.com/Dhivya-Prabhu-K/Modified_Halley-s_Method}}. To evaluate the Hermite and Legendre polynomials and their derivatives, the local Taylor series approach discussed in \cite{GLR07} was used. It is noteworthy that the fourth-order algorithms FOM-H \cite{GST19} and FOM-L \cite{GST21a}, along with the third-order algorithms TOM-H and TOM-L \cite{PSV25}, also employ this approach for the evaluation of the polynomial and its derivative. In this approach, the latest iterate was used as the centre for the Taylor series. The higher-order derivatives in the Taylor series are easily evaluated using a suitable recurrence relation obtained from the associated second-order differential equation.

\subsection{Gauss-Hermite quadrature}
This subsection outlines an algorithm that uses the proposed modified Halley's method \eqref{mscheme} to compute Gauss-Hermite quadrature nodes and weights. Let $x \in \mathbb{R}$ and $n\in \mathbb{N}$. The Hermite polynomial $H_n(x)$ of degree $n$ is known to be a solution of the second-order differential equation given by $y''(x) - 2xy'(x) + 2ny(x) = 0$. From \cite[Table 18.6.1]{OLBC10}, it can be concluded that the zeros of $H_n(x)$ are symmetric about $x=0$. It follows that identifying all the zeros of $H_n(x)$ requires only the determination of its positive zeros. Using \cite[p. 168]{Te96}, it can be concluded that all positive zeros of $H_n(x)$ are confined to the interval $(0, \sqrt{2n+1})$. Consider the function $f(x)=\kappa {\rm e}^{-\frac{x^2}{2}}H_n(x)$, where $\kappa$ is a normalising constant. The functions $H_n(x)$ and $f(x)$ share the same zeros. It is straightforward to confirm that the function $f(x)$ satisfies the subsequent differential equation:
\begin{equation}\label{ch5_eq11}
	f''(x)+r(x)f(x)=0,
\end{equation}
where $r(x)=2n+1-x^2$. Consequently, this subsection uses the proposed modified Halley's method to obtain all zeros of the function $f(x)$ and deduce the nodes and weights of the Gauss-Hermite quadrature. The coefficient function $r(x)$ is positive and monotonically decreasing in the interval $(0, \sqrt{2n+1})$. 
As discussed in \cite{GST19}, the Gauss-Hermite weights $w_i$ corresponding to the nodes $x_i$ are given by the expressions
\begin{eqnarray*}
  w_i &=&\dfrac{\sqrt{\pi}2^{n+1}n!~\kappa^2}{f'(x_i)^2}e^{-x_i^2};~~~~~~
  \frac{1}{2^{n+2}n!~\kappa^2} = \sum_{j=1}^{\lfloor\frac{n}{2}\rfloor} \dfrac{2e^{-x_j^2}x_j^2}{f'(x_j)^2}.
\end{eqnarray*}
\begin{table}[ht]
\centering
\scriptsize
\setlength{\tabcolsep}{3pt} 
\renewcommand{\arraystretch}{1.1} 
\resizebox{\textwidth}{!}{
\begin{tabular}{|c|cc|ccc|ccc|}
\hline
\multirow{2}{*}{\textbf{$\boldsymbol{n}$}} 
& \multicolumn{2}{c|}{\textbf{FOM-H \cite{GST19}}} 
& \multicolumn{3}{c|}{\textbf{TOM-H \cite{PSV25}}} 
& \multicolumn{3}{c|}{\textbf{MHM-H}} \\ \cline{2-9}

& \textbf{T-Iter} & \textbf{A-Time} 
& \textbf{T-Iter} & \textbf{A-Time} & \textbf{MR.Er-F} 
& \textbf{T-Iter} & \textbf{A-Time} & \textbf{MR.Er-F} \\ \hline

1000000 & 508135 & 0.1027 & 524319 & 0.0951 & $2.60\times10^{-16}$ & 508146 & 0.0933 & $1.89\times10^{-16}$ \\
1050000 & 532925 & 0.1040 & 548657 & 0.0996 & $2.22\times10^{-16}$ & 532936 & 0.0970 & $1.81\times10^{-16}$ \\
1100000 & 557731 & 0.1082 & 573045 & 0.1004 & $2.34\times10^{-16}$ & 557742 & 0.0995 & $1.58\times10^{-16}$ \\
1200000 & 607382 & 0.1168 & 621958 & 0.1148 & $2.22\times10^{-16}$ & 607393 & 0.1081 & $1.62\times10^{-16}$ \\
1300000 & 657076 & 0.1261 & 671014 & 0.1251 & $2.22\times10^{-16}$ & 657087 & 0.1165 & $1.53\times10^{-16}$ \\ \hline
\end{tabular}
}
\caption{Performance: FOM-H \cite{GST19} vs TOM-H \cite{PSV25} vs Proposed MHM-H \eqref{mscheme}}
\label{Table_comparison_hermite_ch5}
\end{table}
Table~\ref{Table_comparison_hermite_ch5} displays the total number of iterations and the average time taken by FOM-H \cite{GST19}, TOM-H \cite{PSV25}, and the proposed modified Halley's method, MHM-H \eqref{mscheme}. This data pertains to the computation of all the zeros of selected Hermite polynomials of varying degrees within the range of $[10^{6}, 13\times 10^{5}]$. In Table \ref{Table_comparison_hermite_ch5}, the performance of the proposed modified Halley's method, alongside the third-order method \cite{PSV25}, is compared with that of the fourth-order method \cite{GST19}. As anticipated, FOM-H \cite{GST19} computes all the zeros of $H_n(x)$ with the least number of iterations. Among third-order methods, the proposed modified Halley's method requires significantly fewer iterations than the method presented in \cite{PSV25} to compute all zeros of $H_n(x)$, with better accuracy. Although the proposed method is third-order, the difference in the total number of iterations compared to the fourth-order method \cite{GST19} is relatively minor. Both the third-order methods have better execution times than the fourth-order method. However, the proposed modified Halley's method has the least execution time. This result could be attributed to the computational complexity of repeatedly evaluating functions within the respective iterative methods.

Figure \ref{error-nodes-hermite_ch5} presents the relative errors \eqref{relative_error_ch5} for the methods MHM-H \eqref{mscheme}, FOM-H \cite{GST19}, and TOM-H \cite{GST19}, evaluated at various nodes of Gauss-Hermite quadrature. More specifically, Figure \ref{error-nodes-hermite_ch5} displays the relative error \eqref{relative_error_ch5} across the five thousand positive nodes associated with the Hermite polynomial of degree ten thousand. The figure clearly shows that all methods produce nearly the same level of accuracy. Figure \ref{error-weights-hermite_ch5} presents the relative errors \eqref{relative_error_ch5} for the methods MHM-H \eqref{mscheme}, FOM-H \cite{GST19}, and TOM-H \cite{PSV25}, evaluated at various weights of Gauss-Hermite quadrature. More specifically, Figure \ref{error-weights-hermite_ch5} displays the relative error \eqref{relative_error_ch5} across the five thousand weights corresponding to positive nodes associated with the Hermite polynomial of degree ten thousand. The figure clearly shows that all methods produce good accuracy. Figure \ref{CPU_hermtie_ch5} presents the average CPU time for the methods MHM-H \eqref{mscheme}, FOM-H \cite{GST19}, and TOM-H \cite{PSV25} to compute all zeros of the twenty-five Hermite polynomials of degrees varying in the interval $(12\times10^{5}, 13\times10^{5})$. The average CPU time is calculated by averaging the execution times of the same program run 20 times. The figure clearly demonstrates that the proposed method, MHM-H, is faster than the other methods when computing the zeros of these polynomials.

\begin{figure}[p]
\centering

\begin{subfigure}[t]{0.48\textwidth}
    \centering
    \includegraphics[width=\linewidth]{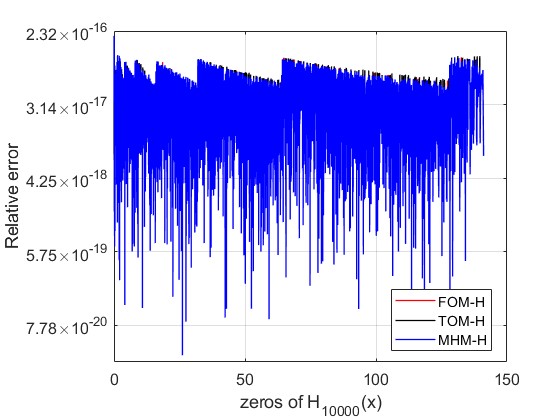}
    \caption{Nodes-Relative error \eqref{relative_error_ch5}}
    \label{error-nodes-hermite_ch5}
\end{subfigure}
\hfill
\begin{subfigure}[t]{0.48\textwidth}
    \centering
    \includegraphics[width=\linewidth]{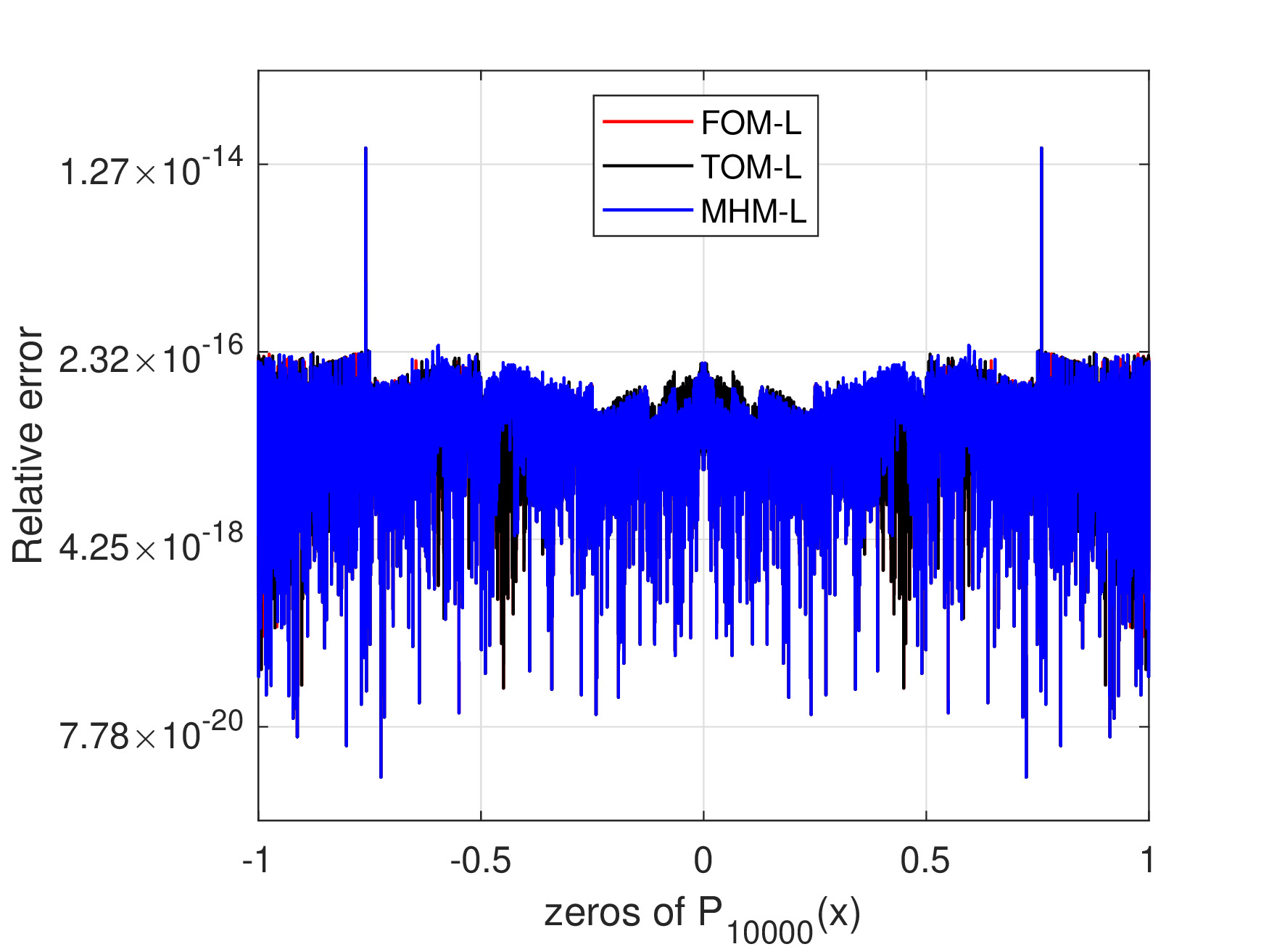}
    \caption{Nodes-Relative error \eqref{relative_error_ch5}}
    \label{error-nodes-legendre_ch5}
\end{subfigure}

\vspace{0.6cm}

\begin{subfigure}[t]{0.48\textwidth}
    \centering
    \includegraphics[width=\linewidth]{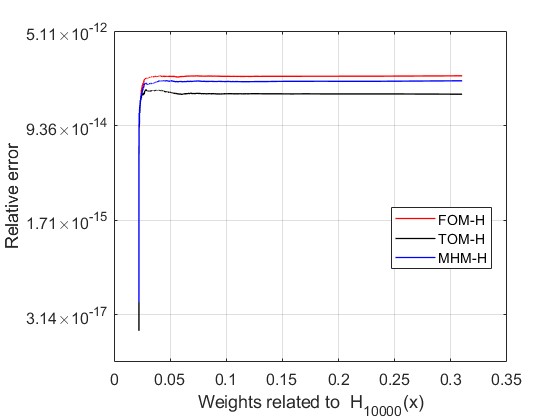}
    \caption{Weights-Relative error \eqref{relative_error_ch5}}
    \label{error-weights-hermite_ch5}
\end{subfigure}
\hfill
\begin{subfigure}[t]{0.48\textwidth}
    \centering
    \includegraphics[width=\linewidth]{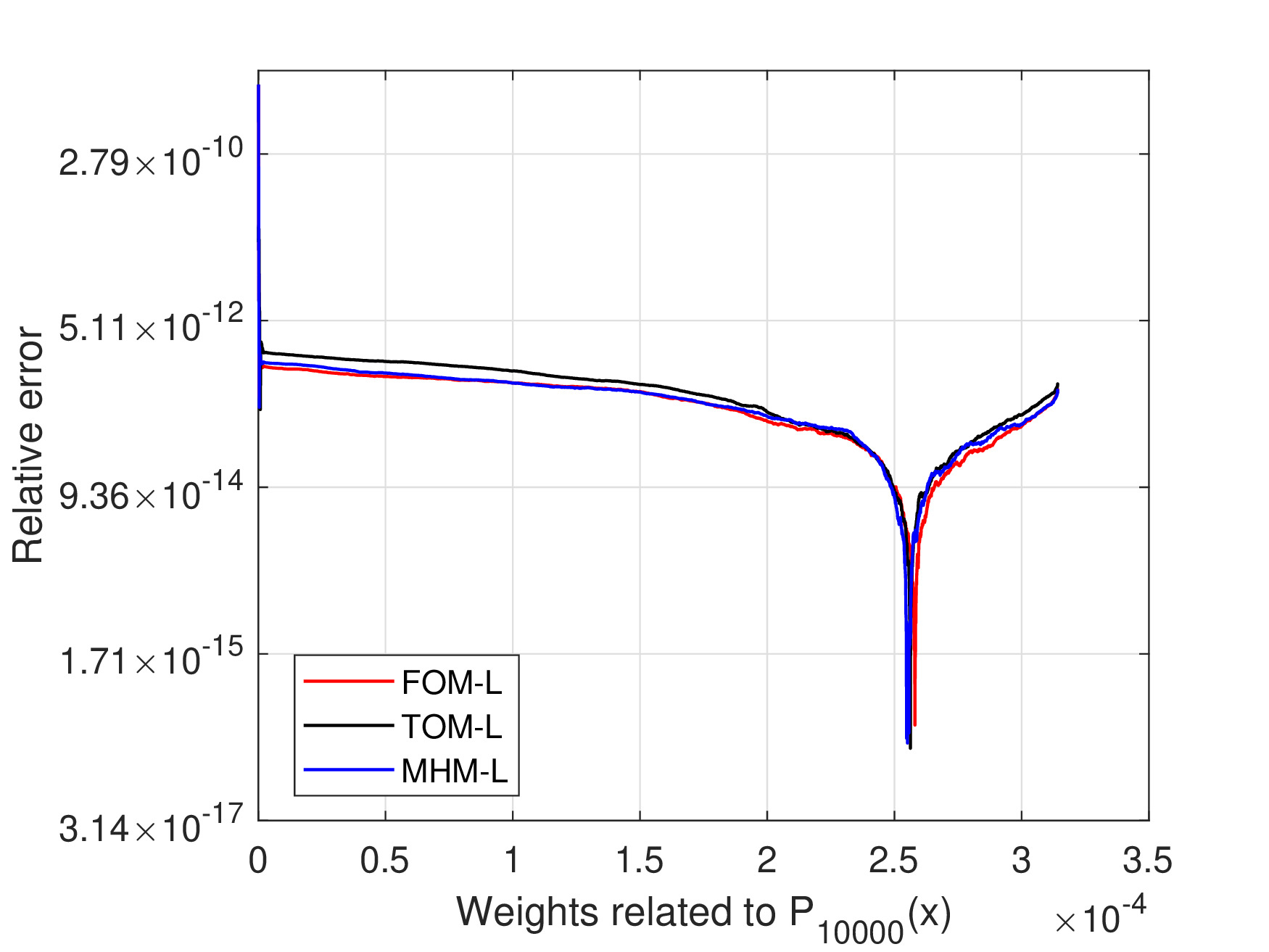}
    \caption{Weights-Relative error \eqref{relative_error_ch5}}
    \label{error-weights-legendre_ch5}
\end{subfigure}

\vspace{0.6cm}

\begin{subfigure}[t]{0.48\textwidth}
    \centering
    \includegraphics[width=\linewidth]{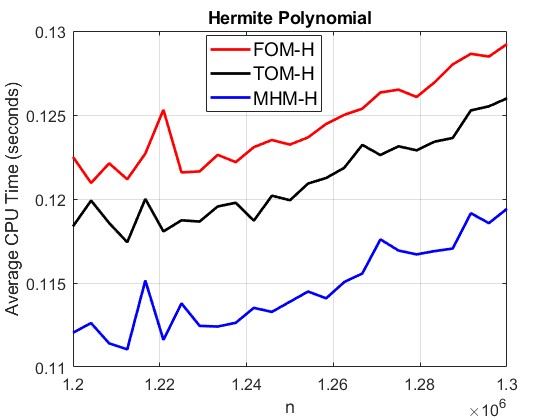}
    \caption{CPU-time-Gauss-Hermite}
    \label{CPU_hermtie_ch5}
\end{subfigure}
\hfill
\begin{subfigure}[t]{0.48\textwidth}
    \centering
    \includegraphics[width=\linewidth]{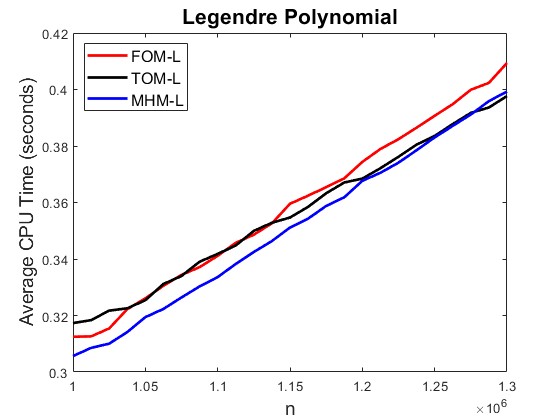}
    \caption{CPU-time-Gauss-Legendre}
    \label{CPU_legendre_ch5}
\end{subfigure}

\caption{FOM-H \cite{GST19} vs TOM-H \cite{PSV25} vs MHM-H  and FOM-L \cite{GST21a} vs TOM-L \cite{PSV25} vs MHM-L}
\label{fig:all_comparisons}
\end{figure}

\subsection{Gauss-Legendre quadrature}
This subsection outlines an algorithm that uses the proposed modified Halley's method \eqref{mscheme} to compute Gauss-Legendre quadrature nodes and weights. Let $x \in (-1, 1)$ and $n\in \mathbb{N}$. The Legendre polynomial $P_n(x)$ of degree $n$ is known to be a solution of the second-order differential equation given by $(1-x^2)y''(x) - 2xy'(x) +n(n+1) y(x) =0$. From \cite[eq. 6.20, p. 143]{Te96}, it can be concluded that the zeros of $P_n(x)$ are symmetric about $x=0$. It follows that identifying all the zeros of $P_n(x)$ requires only the determination of its positive zeros. Consider the change of variable $Y(\zeta) = \lambda P_n(\tanh \zeta)$ as discussed in \cite{GST19}, where $\zeta \in \mathbb{R}$ and $\lambda$ is a normalising constant. Denote $\frac{{\rm d}Y}{{\rm d}\zeta}=\dot{Y}$ and $\frac{{\rm d}^2Y}{{\rm d}\zeta^2}=\ddot{Y}$. It is straightforward to verify that the function $Y(\zeta)$ satisfies the following differential equation (see \cite[Eq. 27]{GST21a}):
\begin{equation}\label{legendre_normal_ch5_1}
\ddot{Y}(\zeta)+R(\zeta)Y(\zeta) = 0,
\end{equation}
where $R(\zeta) = \frac{L^2-1}{4}\text{sech}^2{\zeta}$ and $L = 2n+1$. Consequently, this subsection uses the proposed modified Halley's method to obtain all zeros of the function $Y(\zeta)$ and deduce the nodes and weights of the Gauss-Legendre quadrature. The coefficient function $R(\zeta)$ is positive and monotonically decreasing on $(0, \infty)$. Denote $h(\zeta) = \frac{Y(\zeta)}{\dot{Y}(\zeta)}$. 
The proposed modified Halley's method \eqref{mscheme} to find zero of $h(\zeta)$ with the initial guess $\zeta_0$ is
\begin{equation}\label{iteration_legendre_in_zeta}
    \zeta_{n+1} = \zeta_n-\frac{2h(\zeta_n)}{2+\frac{L^2-1}{4}\text{sech}^2{\zeta_0}~h^2(\zeta_n)}, ~~~n = 0, 1, \cdots.
\end{equation}
Define $y(x) = \lambda \sqrt{1-x^2}P_n(x)$. Let $x(\zeta) = \tanh \zeta$. Then, one can verify that
\begin{equation*}\label{legendre_h_in_x}
    h(\zeta) = b(x(\zeta)) = \frac{y(x(\zeta))}{(1-x^2(\zeta))y'(x(\zeta))+x(\zeta)y(x(\zeta))}.
\end{equation*}
Expressing the above iterative scheme \eqref{iteration_legendre_in_zeta} in terms of $x$, one gets
\begin{eqnarray}
   \tanh^{-1}(x_{n+1}) &=& \tanh^{-1}(x_n)-\frac{2b(x_n)}{2+\frac{L^2-1}{4}(1-x^2_{0})b^2(x_n)} \nonumber\\
    x_{n+1} &=& \frac{x_n-\tanh (k(x_n))}{1-x_n\tanh(k(x_n))},\label{Legendre_iterate_in_x}
\end{eqnarray}
where $k(x) = \frac{2b(x)}{2+\frac{L^2-1}{4}(1-x^2_{0})b^2(x)}$ and $x_0 = \tanh(\zeta_0)$. 
As discussed in \cite{GST21a}, the Gauss-Legendre weights $w_i$ are given by the expressions
\begin{equation}\label{weight_legendre}
   w_i = \frac{2}{y'(x_i)^2}\lambda^2;~~~~\lambda^2 = \left(\left[\frac{1}{y'(x_0)^2}\right]+2\sum_{i=1}^{\lfloor\frac{n}{2}\rfloor}\frac{1}{y'(x_i)^2}\right)^{-1},
\end{equation}
where $x_i$ is the $i^{\mbox{th}}$ positive node of the Gauss-Legendre quadrature. The term $\left[\frac{1}{y'(x_0)^2}\right]$ corresponds to the weight associated with the node $x_0=0$, which occurs only when $n$ is odd. Therefore, this term is omitted when $n$ is even.
\begin{table}[ht]
\centering
\scriptsize
\setlength{\tabcolsep}{3pt} 
\renewcommand{\arraystretch}{1.1} 
\resizebox{\textwidth}{!}{
\begin{tabular}{|c|cc|ccc|ccc|}
\hline
\multirow{2}{*}{\textbf{$\boldsymbol{n}$}} 
& \multicolumn{2}{c|}{\textbf{FOM-L} \cite{GST21a}} 
& \multicolumn{3}{c|}{\textbf{TOM-L} \cite{PSV25}} 
& \multicolumn{3}{c|}{\textbf{MHM-L}} \\ \cline{2-9}

& \textbf{T-Iter} & \textbf{A-Time} 
& \textbf{T-Iter} & \textbf{A-Time} & \textbf{MR.Er-F} 
& \textbf{T-Iter} & \textbf{A-Time} & \textbf{MR.Er-F} \\ \hline

1000000 & 1000004 & 0.3058 & 1130685 & 0.3118 & $5.96\times10^{-11}$ & 1000043 & 0.3013 & $3.33\times10^{-16}$ \\
1050000 & 1050004 & 0.3181 & 1173856 & 0.3200 & $6.82\times10^{-11}$ & 1050044 & 0.3145 & $3.46\times10^{-16}$ \\
1100000 & 1100004 & 0.3333 & 1229162 & 0.3347 & $6.22\times10^{-11}$ & 1100038 & 0.3294 & $3.34\times10^{-16}$ \\
1200000 & 1200004 & 0.3660 & 1321763 & 0.3635 & $7.08\times10^{-11}$ & 1200041 & 0.3606 & $3.38\times10^{-16}$ \\
1300000 & 1300003 & 0.3965 & 1417548 & 0.3901 & $7.64\times10^{-11}$ & 1300035 & 0.3942 & $3.33\times10^{-16}$ \\ \hline

\end{tabular}
}
\caption{Performance: FOM-L \cite{GST21a} vs TOM-L \cite{PSV25} vs Proposed MHM-L \eqref{mscheme}}
\label{comparison_legendre_table_ch5}
\end{table}

Table~\ref{comparison_legendre_table_ch5} displays the total number of iterations and the average time taken by FOM-L \cite{GST21a}, TOM-L \cite{PSV25}, and the proposed modified Halley's method, MHM-L. This data pertains to the computation of all the zeros of selected Legendre polynomials of varying degrees within the range of $[10^{6}, 13\times 10^{5}]$. In Table \ref{comparison_legendre_table_ch5}, the performance of the proposed modified Halley's method, alongside the third-order method \cite{PSV25}, is compared with that of the fourth-order method \cite{GST21a}. As anticipated, FOM-L \cite{GST21a} computes all the zeros of $P_n(x)$ with the least number of iterations. Among third-order methods, the proposed modified Halley's method requires significantly fewer iterations than the method presented in \cite{PSV25} to compute all zeros of $P_n(x)$, with better accuracy. The order of error for TOM-L is $10^{-11}$, attributed to the presence of a few zeros close to 1, as noted in \cite{PSV25}. In contrast, the proposed third-order method does not encounter this problem. Although the proposed method is third-order, the difference in the total number of iterations compared to the fourth-order method \cite{GST21a} is relatively minor. The proposed modified Halley's method has the least execution time. This result could be attributed to the computational complexity of repeatedly evaluating functions within the respective iterative methods.

Figure \ref{error-nodes-legendre_ch5} presents the relative errors \eqref{relative_error_ch5} for the methods MHM-L \eqref{mscheme}, FOM-L \cite{GST21a}, and TOM-L \cite{GST19}, evaluated at various nodes of Gauss-Legendre quadrature. More specifically, Figure \ref{error-nodes-legendre_ch5} displays the relative error \eqref{relative_error_ch5} across the ten thousand nodes associated with the Legendre polynomial of degree ten thousand. The figure clearly shows that all methods produce nearly the same level of accuracy. Figure \ref{error-weights-legendre_ch5} presents the relative errors \eqref{relative_error_ch5} for the methods MHM-L \eqref{mscheme}, FOM-L \cite{GST21a}, and TOM-L \cite{PSV25}, evaluated at various weights of Gauss-Legendre quadrature. More specifically, Figure \ref{error-weights-legendre_ch5} displays the relative error \eqref{relative_error_ch5} across the five thousand weights corresponding to positive nodes associated with the Legendre polynomial of degree ten thousand. The figure clearly shows that all methods produce good accuracy. Figure \ref{CPU_legendre_ch5} presents the average CPU time for the methods MHM-L \eqref{mscheme}, FOM-L \cite{GST21a}, and TOM-L \cite{PSV25} to compute all zeros of the twenty-five Hermite polynomials of degrees varying in the interval $(10^{6}, 13\times10^{5})$. The average CPU time is calculated by averaging the execution times of the same program run 20 times. The figure clearly demonstrates that the proposed method, MHM-L, is faster than the other methods when computing the zeros of these polynomials.

\section{Conclusion}
By solving a Riccati equation related to a second-order differential equation using the trapezoidal rule, a third-order iterative method is developed to identify the zeros of the solutions to the Riccati equation. This third-order method aligns with the well-known Halley's method for finding the zeros of the solutions to the associated second-order differential equation. A modification of Halley's method ensures third-order convergence and requires functional evaluation similar to that of Newton's method for a class of functions. Based on the behaviour of the coefficient function of the Riccati ODE, global convergence results for both Halley's method and the modified method have been established. Algorithms have been developed using the modified method to compute the nodes and weights for Gauss–Legendre and Gauss–Hermite quadrature. Numerical experiments illustrate the efficiency of the proposed algorithms. A comparative study with recent algorithms in the literature is also included. Finally, regions where the proposed algorithms surpass existing methods in CPU time have been identified and presented.
\section*{Declaration of generative AI and AI-assisted technologies in the writing process}
    During the preparation of this work, the authors used Grammarly to enhance readability. After using this tool/service, the authors reviewed and edited the content as needed and take full responsibility for the content of the published article.

\end{document}

%% file: ex_shared.tex
\usepackage{lipsum}
\usepackage{amsfonts}
\usepackage{graphicx}
\usepackage{epstopdf}
\usepackage{algorithm}
\usepackage{algpseudocode}
\ifpdf
  \DeclareGraphicsExtensions{.eps,.pdf,.png,.jpg}
\else
  \DeclareGraphicsExtensions{.eps}
\fi

\newsiamremark{remark}{Remark}
\newsiamremark{hypothesis}{Hypothesis}
\crefname{hypothesis}{Hypothesis}{Hypotheses}
\newsiamthm{claim}{Claim}
\newsiamremark{fact}{Fact}
\crefname{fact}{Fact}{Facts}
\newsiamremark{example}{Example}

\headers{Modified Halley's method for zeros of solution of ODE{\lowercase{s}}}{K. D. Prabhu, S. Singh and V. A. Vijesh}

\title{Modified Halley's method for the computation of zeros of solutions to second-order ODE{\lowercase{s}} with applications\thanks{Submitted to the editors DATE. 29/05/2026
\funding{The first author, Dhivya Prabhu K, gratefully acknowledges the financial support from the Council of Scientific and Industrial Research, India (Grant No. 09/1022(11054)/2021-EMR-I).}}}

\author{
Dhivya Prabhu K\thanks{Department of Mathematics, Indian Institute of Technology Indore, 
Indore 453552, India (\texttt{dhivyanslm@gmail.com}, \texttt{snjvsngh@iiti.ac.in}, \texttt{vijesh@iiti.ac.in}).}
\and
Sanjeev Singh$^\dagger$
\and
Antony Vijesh V$^\dagger$
}

\usepackage{amsopn}


%% file: ex_article.bbl
\begin{thebibliography}{10}
\bibitem{AO08}
{\sc R. P. Agarwal and D. O'Regan}, {\em An Introduction to Ordinary Differential Equations}, Springer, 2008.

	\bibitem{BMF12}
	{\sc I. Bogaert, B. Michiels, and J. Fostier}, {\em $\mathcal{O}(1)$ computation of Legendre polynomials and Guass-Legendre nodes and weights for parallel computing}, SIAM J. Sci. Comput., 34 (2012), pp. c83-c101. 
    
        
    \bibitem{Br17}
		{\sc J. Bremer}, {\em On the numerical calculation of the roots of special functions satisfying second order ordinary differential equations}, {SIAM J. Sci. Comput.}, 39 (2017), A55-A82.
        
        
        

    
	
	
	
	\bibitem{GS03}
	{\sc	A. Gil and J. Segura }, \textit{Computing the zeros and turning points of solutions of second order homogeneous linear ODEs}, SIAM J. Numer. Anal.,  41 (2003), pp. 827-855, \href{https://doi.org/10.1137/S0036142901392754}{https://doi.org/10.1137/S0036142901392754}.
    

    
	
	
	
	\bibitem{GST19}
	{\sc A. Gil, J. Segura, and N. M. Temme}, \textit{Fast, reliable and unrestricted iterative computation of Gauss-Hermite and Gauss-Laguerre quadratures}, Numer. Math. 3 (2019), pp. 649-682.
	
	
	\bibitem{GST21a} 	
	{\sc A. Gil, J. Segura, and N.M. Temme}, {\em Fast and reliable high-accuracy computation of Gauss-Jacobi quadrature,} {Numer. Algorithms.,} 87 (2021), pp. 1391-1419. 
	\href{https://doi.org/10.1007/s11075-020-01012-6}{https://doi.org/10.1007/s11075-020-01012-6.}
	\bibitem{GLR07}
	{\sc A. Glaser, X. Liu, and V. Rokhlin,} {\em A fast algorithm for the calculation of the root of special functions}, SIAM J. Sci. Comput., 29 (2007), pp. 1420-1438.
	
	\bibitem{HT13}
	{\sc N. Hale and A. Townsend}, {\em Fast and accurate computation of Gauss-Legendre and Gauss-Jacobi quadrature nodes and weights}, SIAM J. Sci. Comput., 35 (2013), pp. A652-A674.
	



        



    
	
	\bibitem{Me97}
	{\sc A. Melman,} \textit{Geometry and convergence of Euler's and Halley's method}, SIAM Rev., 39 (1997), pp. 728-735.
    


\bibitem{Ol50}
{\sc F.~W.~J. Olver}, {\em A new method for the evaluation of zeros of Bessel functions and of other solutions of second-order differential equations}, Proc. Cambridge Philos. Soc., {46} (1950), pp. 570--580.

\bibitem{OLBC10}
	{\sc F.W.J. Olver, D.W. Lozier, R.F. Boisvert, and C.W. Clark}, {\em NIST Handbook of Mathematical Functions}, Cambridge University Press, Cambridge, UK (2010).


    
    
    \bibitem{PSV25}
    {\sc K. D. Prabhu, S. Singh, and V. A. Vijesh},  {\em Efficient third-order iterative algorithms for computing zeros of special functions}, \href{https://doi.org/10.48550/arXiv.2601.04148}{https://doi.org/10.48550/arXiv.2601.04148}, Under review.


    
	
	
	\bibitem{Se02}
	{ \sc J. Segura}, \textit{The zeros of special functions from a fixed point method}, SIAM J. Numer. Anal., 40 (2002), pp. 114-133, \href{https://doi.org/10.1137/S0036142901387385}{https://doi.org/10.1137/S0036142901387385}.
	
	\bibitem{Se10}
	{\sc J. Segura}, \textit{Reliable computation of the zeros of solutions of second order linear ODEs using a fourth order method}, SIAM J. Numer. Anal., 48 (2010), pp. 452–469, \href{https://doi.org/10.1137/090747762}{https://doi.org/10.1137/090747762}.
	
\bibitem{SV90}
{\sc R. Spigler, and M. Vianello}, \textit{A numerical method for evaluating zeros of solutions of second-order linear differential equations}, Math. Comp., 55 (1990), no.~192, pp. 591--612.

	\bibitem{Te96}
	{\sc N. M. Temme}, {\em Special functions: An Introduction to the Classical Functions of Mathematical Physics}, John Wiley \& Sons, (2011).
	
	\bibitem{TTO16}
	{\sc A. Townsend, T. Trogdon, and S. Olver}, {\em Fast computation of Gauss quadrature nodes and weights on the whole real line}, IMA J. Numer. Anal., 36 (2016), pp. 337-358, \href{https://doi.org/10.1093/imanum/drv002}{https://doi.org/10.1093/imanum/drv002}.
	

    
\end{thebibliography}
